\documentclass[10pt,a4paper]{amsart}
\usepackage{amsfonts}
\usepackage{amsthm}
\usepackage{amsmath}
\usepackage{amscd}
\usepackage[utf8]{inputenc}
\usepackage{t1enc}
\usepackage[mathscr]{eucal}
\usepackage{indentfirst}
\usepackage{graphicx}
\usepackage{graphics}
\usepackage{pict2e}
\usepackage{epic}
\usepackage{shuffle} 
\usepackage[table,xcdraw]{xcolor} 
\numberwithin{equation}{section}
\usepackage[margin=2.9cm]{geometry}
\usepackage{epstopdf}

\theoremstyle{plain}
\newtheorem{theorem}{Theorem}[section]
\newtheorem{lemma}[theorem]{Lemma}
\newtheorem{corollary}[theorem]{Corollary}
\newtheorem{proposition}[theorem]{Proposition}

\theoremstyle{definition}

\newtheorem{conjecture}[theorem]{Conjecture}
\newtheorem{remark}[theorem]{Remark}
\newtheorem{?}[theorem]{Problem}
\newtheorem{example}[theorem]{Example}

\DeclareMathOperator{\Li}{Li}

\begin{document}
	
	\title{Evaluation of one-dimensional polylogarithmic integral, with applications to infinite series}

	\author[Kam Cheong Au]{Kam Cheong Au}
	
	\address{Rheinische Friedrich-Wilhelms-Universität Bonn \\ Mathematical Institute \\ 53115 Bonn, Germany} 
	
	\email{s6kmauuu@uni-bonn.de}
	\subjclass[2010]{Primary: 	11M32. Secondary: 33C20}

	\keywords{Multiple zeta values, Colored polylogarithm, Multiple polylogarithm, Multiple $L$-functions, Fourier-Legendre expansion, Definite integral, Harmonic number}

	\begin{abstract} We give systematic method to evaluate a large class of one-dimensional integral relating to multiple zeta values (MZV) and colored MZV. We also apply the technique of iterated integrals and regularization to elucidate the nature of some infinite series involving binomial coefficients. This technique can be applied to many Apéry-type infinite sums. 
	\end{abstract}
	
	\maketitle
	
	\section{Introduction}
	In this paper, we will denote
	$$\zeta(s_1,\cdots,s_k) = \sum_{s_1>\cdots>s_k\geq 1} \frac{1}{n_1^{s_1} \cdots n_k^{s_k}}\qquad L_{s_1,\cdots,s_k}(a_1,\cdots,a_k) = \sum_{s_1>\cdots>s_k\geq 1}\frac{a_1^{n_1}\cdots a_k^{n_k}}{n_1^{s_1} \cdots n_k^{s_k}}$$
	to be the \textit{multiple zeta function} and \textit{colored polylogarithm} (or multiple $L$-values) respectively. For these two functions, $k$ is called the \textit{length} and $s_1+\cdots+s_k$ is called the \textit{weight}. We will also use the following common notation for alternating MZV: for example $\zeta(3,1,\bar{2},\bar{1}) = L_{3,1,2,1}(1,1,-1,-1)$. \par
	
	When $a_i$ are $N$-th roots of unity, $s_i$ are positive integers and $(a_i, s_i) \neq (1,1)$, $L_{s_1,\cdots,s_k}(a_1,\cdots,a_k)$ is called a colored multiple zeta values (CMZV) of weight $s_1+\cdots+s_k$ and \textit{level} $N$. Denote the $\mathbb{Q}$-span of weight $n$ and level $N$ CMZVs by $\textsf{CMZV}^N_n$. 
	
	There have been a lot of researches on $\mathbb{Q}$-dimensions spanned by CMZV as well as linear relations between them (\cite{arakawa2004multiple}, \cite{racinet2002doubles}, \cite{deligne2010groupe}, \cite{ZhaoStandard}, \cite{ZhaoStandard}). 
	
	One major focus of the paper (Section 3-5) is to present a method (Theorem \ref{intmaintheo}) to calculate a large class of one-dimensional integral involving ordinary polylogarithm and generalized polylogarithm. The main idea is to represent the integrand as an iterated integral, but there are some subtleties. As an immediate consequence, some highly nontrivial result for infinite series are obtained. For example, 
	$$\sum\limits_{n = 2}^\infty  {\frac{{H_{n - 1}^{(2)}}}{{{n^3}}}\left[ {{2^n}{\binom{2n}{n}^{ - 1}}} \right]} = \frac{\pi ^3 C}{24}-\pi  \beta(4)-\frac{3 \pi ^2 \zeta (3)}{128}+\frac{527 \zeta (5)}{256}+\frac{1}{384} \pi ^4 \log (2)$$
	\small \begin{multline*}\sum_{n=1}^\infty \frac{1}{n^5 2^n \binom{3n}{n}} = 4 \pi  \Im\left(\text{Li}_4\left(\frac{1}{2}+\frac{i}{2}\right)\right)+3 \pi  \beta(4)-\frac{51 \text{Li}_5\left(\frac{1}{2}\right)}{2} -15 \text{Li}_4\left(\frac{1}{2}\right) \log (2)+\frac{\pi ^2 \zeta (3)}{4}+\frac{9 \zeta (5)}{2}\\ \qquad -3 \zeta (3) \log ^2(2) -\frac{97}{240} \log ^5(2)+\frac{41}{144} \pi ^2 \log ^3(2)-\frac{61}{960} \pi ^4 \log (2) \end{multline*}
	\normalsize{here} $C = \beta(2)$ denotes the Catalan constant, $\beta$ is the Dirichlet Beta function. The last series was conjectured by \cite[p.~27-28]{borwein2004experimentation}. Some of the series have been proved via complicated but ingenious manipulations in \cite{ZhaoMingHao1}, \cite{ZhaoMingHao2}. Techniques of this article inspire many related efforts by various authors (\cite{au2022iterated}, \cite{zhou2022hyper}, \cite{zhou2023sun}, \cite{xu2023dirichlet}, \cite{au2022multiple}) to tackle systemically many Apéry-type sums-conjectures due to Z.W. Sun (\cite{sun2013products}, \cite{sun2021book}). \par
	Another line of results (Section 4.1 - 4.3) is the following: let $a_n = 4^{-n}\binom{2n}{n}$, $H_n^{s_1,\cdots,s_k} = \sum_{n\geq n_1>\cdots>n_k\geq 1}(n_1^{s_1}\cdots n_k^{s_k})^{-1}$, set $S = s+s_1+\cdots+s_k$, then: 
	$$\sum_{n=1}^\infty \frac{H_n^{s_1,\cdots,s_k}}{n^s}a_n^{\pm 1} \in \textsf{CMZV}^2_{S} \qquad \sum_{n=1}^\infty \frac{H_n^{s_1,\cdots,s_k}}{n^s}a_n^{-2} \in \textsf{CMZV}^4_{S}\qquad \sum_{n=1}^\infty \frac{H_n^{s_1,\cdots,s_k}}{n^s}a_n^2 \in \frac{1}{\pi}\textsf{CMZV}^4_{S+1}$$
	~\\[0.02in]
	As immediate consequence of results in Section 3, some neat evaluation of definite integral can be obtained:
	\small $$\int_0^1 \frac{\Li_2 (-\frac{4x}{(1-x)^2}) \Li_3 (1-x^2)}{x} dx = -4 \pi ^2 \text{Li}_4\left(\frac{1}{2}\right)+\frac{7 \zeta (3)^2}{8}-\frac{93}{2} \zeta (5) \log (2)+\frac{139 \pi ^6}{3360}-\frac{1}{6} \pi ^2 \log ^4(2)+ \frac{1}{6} \pi ^4 \log ^2(2) $$
	\small \begin{multline*}
		\int_0^1 \frac{\log^2 (1-x) \log^2 x \log^3(1+x)}{x} dx = -168 \text{Li}_5\left(\frac{1}{2}\right) \zeta (3)+96 \text{Li}_4\left(\frac{1}{2}\right){}^2-\frac{19}{15} \pi ^4 \text{Li}_4\left(\frac{1}{2}\right)+\\ 12 \pi ^2 \text{Li}_6\left(\frac{1}{2}\right)+8 \text{Li}_4\left(\frac{1}{2}\right) \log ^4(2)-2 \pi ^2 \text{Li}_4\left(\frac{1}{2}\right) \log ^2(2)+12 \pi ^2 \text{Li}_5\left(\frac{1}{2}\right) \log (2)+\frac{87 \pi ^2 \zeta (3)^2}{16}+\\ \frac{447 \zeta (3) \zeta (5)}{16}+\frac{7}{5} \zeta (3) \log ^5(2)-\frac{7}{12} \pi ^2 \zeta (3) \log ^3(2)-\frac{133}{120} \pi ^4 \zeta (3) \log (2)-\frac{\pi ^8}{9600}+\frac{\log ^8(2)}{6}- \\ \frac{1}{6} \pi ^2 \log ^6(2)-\frac{1}{90} \pi ^4 \log ^4(2)+\frac{19}{360} \pi ^6 \log ^2(2)
	\end{multline*}
	\normalsize{the} last expression is remarkable because it lacks level 2 weight 8 CMZV with higher length. 
	
	In Section 2, we quickly recalls relevant notations on Hoffman-Racinet algebra which will be used in subsequent sections. The last part of the section deals with the reduction of CMZV into more elementary constants. In \cite{ZhaoStandard} and \cite{zhao2016multiple}, although ways to obtaining relations are detailed, only linear relations satisfied by CMZVs on a particular weight and level are considered. Using the method in these two papers, we calculated a complete reduction for level $4$ weight $5$ CMZVs.

	\section{Preliminaries}
	\subsection{Iterated integral}
	We quickly assemble required facts of iterated integral (\cite{chen1971algebras}, \cite{gil2017multiple}). Let functions $f_i(t)$ defined on $[a,b]$, define inductively
	$$\int_a^b f_1(t) dt \cdots f_n(t) dt = \int_a^b f_1(u) du \cdots f_{n-1}(u) \int_a^u f_n(t) dt $$
	When $r=1$, this is the usual definite integral of $\int_a^b f_1(t) dt$. When $r=0$, define its value to be $1$. \par
	The definition can be extended to manifold. Let $\gamma: [0,1]\to M$ a path on a manifold $M$, $\omega_1,\cdots,\omega_n$ be differential $1$-forms on $M$. Then
	$$\int_\gamma \omega_1\cdots \omega_n := \int_0^1 f_1(t) dt \cdots f_r(t) dt$$
	with $\gamma^\ast \omega_i = f_i(t) dt$ being the pullback of $\omega$. Then if $f: N\to M$ is a differentible map between two manifolds $N$ and $M$, 
	\begin{equation}\label{itintpullback}
		\int_{f\circ \gamma} \omega_1\cdots \omega_n = \int_\gamma f^\ast \omega_1 \cdots f^\ast\omega_n
	\end{equation}
	
	\begin{proposition}
		Iterated integral enjoys the following properties:
		$$\int_\gamma \omega_1\cdots \omega_n = (-1)^n \int_{\gamma^{-1}} \omega_n\cdots \omega_1$$
		where $\gamma^{-1}$ is the reverse path of $\gamma$. 
		$$\int_{\gamma_1 \gamma_2} \omega_1\cdots \omega_n = \sum_{r=0}^n \int_{\gamma_1} \omega_1\cdots \omega_r \int_{\gamma_2} \omega_{r+1}\cdots \omega_n$$
		where $\gamma_2(1) = \gamma_1(0)$, here $\gamma_1\gamma_2$ means composition of two paths, first $\gamma_2$, then $\gamma_1$. 
		\begin{equation}\label{itintshuffle}\int_{\gamma} \omega_1\cdots \omega_n \int_{\gamma} \omega_{n+1}\cdots \omega_{n+m} = \sum_{\substack{\sigma\in S_{n+m} \\ \sigma(1)<\cdots<\sigma(n) \\ \sigma(n+1)<\cdots<\sigma(n+m)}} \int_\gamma \omega_{\sigma^{-1}(1)} \cdots \omega_{\sigma^{-1}(n+m)}\end{equation}
		the last sum is over certain elements of symmetric group $S_{n+m}$, it can also be viewed as shuffle product between $\omega_1\cdots \omega_n$ and $\omega_{n+1}\cdots \omega_{n+m}$, as defined in next subsection.
	\end{proposition}
	
	For use in Section 4, we define another kind of iterated integral. Let $\omega_1,\cdots,\omega_n$ be differential forms on $[0,1]$, for $0\leq a, b\leq 1$, we will put a bar over a differential form to indicate limit of integration should go from $x$ to $1$, rather than from $0$ to $x$. For example, if $\omega_i = f_i(x) dx$, 
	$$\int_a^b \omega_1 \overline{\omega_2} \omega_3 \overline{\omega_4}= \int_a^b f_1(x_1) dx_1 \int_{x_1}^1 f_2(x_2) dx_2 \int_0^{x_2} f_3(x_3) dx_3 \int_{x_3}^1 f_4(x_4) dx_4$$
	A bar will never be put on the first differential form $\omega_1$, so this notation will not cause confusion. By writing $\int_{x}^1 = \int_0^1 - \int_0^{x}$, iterated integral with bars can be converted into linear combination of those without bars (i.e. all limits of integration except first one are from $0$ to $x$). 
	
	\subsection{Hoffman-Racinet algebra}
	Let $X$ be a set, $\mathbb{Q}\langle X\rangle$ be the free non-commutative polynomial algebra over $\mathbb{Q}$ generated over $X$. Treating $X$ as alphabet, let $X^*$ be the set of words over $X$. \par
	Define a binary operation $\shuffle$ on $\mathbb{Q}\langle X\rangle$ via:
	$$w \shuffle 1 = 1 \shuffle w = w \qquad x w \shuffle y v = x(w\shuffle yv)+y(xu\shuffle v)$$
	for $w,v\in X^\ast, x,y\in X$. Then distribute $\shuffle$ over addition and scalar multiplication. The shuffle product is commutative and associative. \par 
	Using shuffle product, the last property $\ref{itintshuffle}$ of iterated integral can be written as $$\int_{\gamma} \omega_1\cdots \omega_n \int_{\gamma} \omega_{n+1}\cdots \omega_{n+m} = \int_\gamma \omega_1\cdots \omega_n \shuffle \omega_{n+1}\cdots \omega_{n+m}$$
	~\\[0.02in] 
	Now we specialize to the situation of CMZV. Fix a positive integer $N$, $\mu = \exp(2\pi i /N)$, set $$a = \frac{dt}{t}, b_i = \frac{dt}{\mu^{-i}-t}, z_{k,i} = a^{k-1}b_i \qquad i=0,\cdots,N-1$$
	Let $X = \{a,b_0,\cdots,b_{N-1}\}$, set $\mathfrak{A}^N = \mathbb{Q}\langle X \rangle$, $\mathfrak{A}^N_1$ be the subalgebra of $\mathbb{Q}\langle X\rangle$ generated by $z_{k,i}$;  $\mathfrak{A}^N_0$ be the subalgebra of $\mathfrak{A}^N_1$ generated by words not beginning with $b_0$ and not ending with $a$.
	then if $(s_1, i_1)\neq (1,0)$,
	\begin{equation}\label{toitint} L_{s_1,\cdots,s_n}(\mu^{i_1},\cdots,\mu^{i_n}) = \int_0^1 z_{s_1,i_1}z_{s_2,i_1+i_2}\cdots z_{s_n, i_1+i_2+\cdots i_n}\end{equation}
	
	Let $$\text{\L} (z_{s_1,i_1}\cdots z_{s_n,i_n}) = L_{s_1,\cdots,s_n}(\mu^{i_1},\mu^{i_2-i_1},\cdots, \mu^{i_n-i_{n-1}})$$
	then $\text{\L}(w) = \int_0^1 w$ for $w\in \mathfrak{A}^N_0$. We extend $\text{\L}$ linearly to $\mathbb{Q}\langle X\rangle$. \par
	
	Next we define the \textit{stuffle product} $\ast$ on $\mathfrak{A}^N_1$. Let $j\in \mathbb{Z}$, define $$\tau_j(z_{s_1,i_1}\cdots z_{s_n,i_n}) = z_{s_1,j+i_1}\cdots z_{s_n,j+i_n}$$
	here the second subscript of $z_{s,i}$ is considered modulo $N$. Then 
	\begin{equation}\label{stuffle} z_{s,j} u \ast z_{t,k}v = z_{s,j} \tau_j(\tau_{-j}(u) z_{t,k}v) + z_{t,k} \tau_k(z_{s,j}u\ast \tau_{-k}v) + z_{s+t,j+k}\tau_{j+k} (\tau_{-j}(u)\ast \tau_{-k}(v))\end{equation}
	Then distribute $\ast$ over addition and scalar multiplication. The stuffle product is commutative and associative. \par 
	It can be shown that, for $u,v \in \mathfrak{A}^N_0$,
	$$\text{\L}(u\shuffle v) = \text{\L}(u)\text{\L}(v) = \text{\L}(u\ast v)$$
	The above equality called \textit{finite double shuffle relation}, there is also a regularized version, but we do not state it here. Distribution (regularized or finite) also provides new relation \cite{ZhaoStandard}. When $N$ is not a prime power, Zhao (\cite{zhao2008multiple}, \cite{ZhaoStandard}) also conjectures that these methods do not exhaust all relations between CMZVs. When $N=4$, Zhao \cite{zhao2008multiple} discovers a way to generate these exceptional relations, he used these relations to reach bound predicted by Deligne $\textsf{CMZV}^4_l$, $l=2,3,4$. The author of this paper performed additional computation, Deligne's bound is also attained for $\textsf{CMZV}^4_5$.
	
	\begin{example}[Harmonic number]
		For future reference, we record an example here. Let $N=1$, $M$ be a fixed positive integer, $x_0 = a, x_1 = b_0, z_{k,i} = y_k$, so $y_k = x_0^{k-1} x_1$. For $w=y_{s_1}\cdots y_{s_k}$, define
		$$H_M(w) = \sum_{M\geq n_1 >\cdots > n_k \geq 1} \frac{1}{n_1^{s_1} \cdots n_k^{s_k}}$$
		and extend $H_M$ linearly to $\mathbb{Q}\langle X \rangle$, then it can be shown that $H_M(w \ast v) = H_M(w) H_M(v)$ for $w,v\in \mathfrak{A}^1_1$. For empty word, define $H_M(1) = 1$ for $M\geq 0$; for non-empty word $w$, $H_0(w) = 0$. \par
		The above definition can be generalized to arbitrary level: let $w = z_{s_1,i_1}\cdots z_{s_n,i_n}\in \mathfrak{A}^N_1$, define
		$$H_M(w) = \sum_{M\geq n_1 >\cdots > n_k \geq 1} \frac{\mu^{i_1}\mu^{i_2-i_1}\cdots \mu^{i_n-i_{n-1}}}{n_1^{s_1} \cdots n_k^{s_k}}\qquad \mu = \exp(2\pi i /N)$$
		then $H_M(w \ast v) = H_M(w) H_M(v)$ still holds for $u,v\in \mathfrak{A}^N_1$ (see \cite{bigotte2002lyndon}). 
	\end{example}
	
	\begin{proposition}
		Let $w = z_{s_1,i_1}\cdots z_{s_n,i_n}\in \mathfrak{A}^N_1$, then there exists positive integers $s_i$, $a_i \in \mathbb{C}$ such that
		$$H_M(w) = \sum a_i (\log M +\gamma)^{s_i} + c + o(1)\qquad M\to\infty$$
		with $\gamma$ the Euler-Mascheroni constant and $c\in \textsf{CMZV}^N_{s_1+\cdots+s_n}$. 
	\end{proposition}
	\begin{proof}
		The existence of asymptotic expansion of this form is proved in (\cite{racinet2002doubles}, Section 2). To show that $c$ is a CMZV of level $N$, use the fact that $(\mathfrak{A}^N_1,\ast)$ is a commutative polynomial algebra over $(\mathfrak{A}^N_0,\ast)$ generated by $z_{1,0}$, (\cite{hoffman1997algebra}, \cite{ihara2006derivation}) and $H_M(z_{1,0}) = 1+\frac{1}{2}+\cdots+\frac{1}{M} = \log M + \gamma + O(1/M)$. 
	\end{proof}
	
	\subsection{Regularization}
	We wish to extend the domain of $\text{\L}$ to all of $\mathbb{Q}\langle X\rangle$. Extend the definition as follows: let $k,m,n\geq 0$ be integers, $\xi_i \in X, \xi_1 \neq b_0, \xi_k \neq a$, set $\xi_1 \cdots \xi_k a^n = \xi_1 \cdots \xi_q$, 
	\begin{equation}\label{reg}\text{\L}(b_0^m \xi_1 \cdots \xi_k a^n) = \begin{cases} 0 \qquad &\text{ if } mn=k=0 \\
			\text{\L}(\xi_1 \cdots \xi_k)  \qquad &\text{ if } m=n=0 \\
			-\frac{1}{m}\sum_{i=1}^q \text{\L}(b_0^{m-1} \xi_1 \cdots \xi_i b_0 \xi_{i+1}\cdots \xi_q)  \qquad &\text{ if } m>0 \\
			-\frac{1}{n}\sum_{i=1}^k \text{\L}(\xi_1 \cdots \xi_{i-1} a \xi_{i+1}\cdots \xi_k a^{n-1})  \qquad &\text{ if } m=0,n>0 
	\end{cases}\end{equation}
	\begin{theorem}\label{regularization}
		Let $w\in \mathbb{Q}\langle X\rangle$, then there exists positive integers $s_i, t_i$ and $a_i,b_i,c\in \mathbb{C}$ such that
		$$\int_\alpha^\beta w = \sum a_i \log^{s_i}(\alpha) + \sum b_i \log^{t_i}(1-\beta) + c + o(1) \qquad \alpha\to 0^+,\beta\to 1^-$$
		the above method of extending \textup{\L} will make $c = \text{\L}(w)$.
	\end{theorem}
	\begin{proof}
		(\ref{reg}) is equivalent to saying that certain formal sum is grouplike in the bialgebra $\mathfrak{A}^N$, it is also the unique lift from certain grouplike element in $\mathfrak{A}^N_1$, satisfying the initial condition that nullifies the divergent part of iterated integral. Core ideas of such argument can be found in (\cite[Chap.~2,~13]{zhao2016multiple}, \cite{racinet2002doubles}). 
	\end{proof}
	
	\begin{remark}
		The $c$ in the theorem will still be denoted by $\int_0^1 w$, even though the integral might not make sense from the traditional perspective. For example, $\int_0^1 \frac{dx}{x} = \int_0^1 a = \text{\L}(a) = 0$.
	\end{remark}
	
	\begin{example} [Generalized polylogarithm] Let $X = \{x_0, x_1\}$, $x_0 = dx/x, d_1 = dx/(1-x)$, for any word $w$ formed from $X$, define inductively:
		$$\Li_w(x) = \begin{cases}\frac{\log^n x}{n!} \qquad &\text{ if }w=x_0^n \\ 
			\int_0^x x_i \Li_v(x) \qquad &\text{ if }w=x_i v \end{cases}$$
		$\Li_w$ is then extended linearly to all element in $\mathbb{Q}\langle x_0,x_1\rangle$.
		If $w = x_0^{s_1-1}x_1 \cdots x_0^{s_k-1}x_1 = x_0^{s_1-1}x_1 v\in \mathfrak{A}^1_1$, we have
		$$\Li_{s_1,\cdots,s_k}(x) = \Li_w(x) = \sum_{n_1>\cdots>n_k\geq 1} \frac{x^{n_1}}{n_1^{s_1}\cdots n_k^{s_k}} = \sum_{n=1}^\infty H_{n-1}(v) \frac{x^n}{n^{s_1}}$$
		For any $y\in \mathbb{Q}\langle x_0,x_1\rangle$, $\Li_y(1) = \text{\L}(y)$ with regularization if necessary.
		
	\end{example}

	\subsection{Some low level, low weight reduction}
	Let $d(w,N)$ be the dimension of $\textsf{CMZV}^N_w$, a deep result due to Deligne and Goncharov \cite{deligne2010groupe} provides an upper bound of $d(w,N)$: \begin{theorem}
		Let $D(w,N)$ be defined by
		$$1+\sum_{w=1}^\infty D(w,N) t^w = \begin{cases}(1-t^2-t^3)^{-1} \qquad &\text{ if } N = 1 \\ (1-t-t^2)^{-1} \qquad &\text{ if } N =2 \\ (1-at+bt^2)^{-1} \qquad &\text{ if } N \geq 3 \end{cases}$$
		where $a=\varphi(N)/2 + \nu(N), b=\nu(N)-1$. Here $v(N)$ denote number of distinct prime factors of $N$ and $\varphi$ is the Euler totient function. Then $d(w,N)\leq D(w,N)$.
	\end{theorem}
	
	Let $c(w,N)$ be defined by $$\sum_{w\geq 0} c(w,N) t^w = \log\left(1+\sum_{w=1}^\infty d(w,N) t^w \right)$$
	Assuming the algebra of CMZV of level $N$ forms a graded (with respect to weight) free algebra, then $$\tilde{d}(w,N) = \dim_\mathbb{Q} \widetilde{\textsf{CMZV}}^N_w = \sum_{k\mid w}\frac{\mu(k)}{k}c(\frac{w}{k},N)$$
	with $\mu$ the Möbius function. Moreover assume Deligne's bound is tight, then the following table gives $\tilde{d}(w,N)$ for small $(w,N)$:
	
	\begin{table}[h]
		\begin{tabular}{l|cccccccccccccccc}
			\cline{1-1}
			\multicolumn{1}{|l|}{$N \string\ w$} & \cellcolor[HTML]{C0C0C0}1 & \cellcolor[HTML]{C0C0C0}2 & \cellcolor[HTML]{C0C0C0}3 & \cellcolor[HTML]{C0C0C0}4 & \cellcolor[HTML]{C0C0C0}5 & \cellcolor[HTML]{C0C0C0}6 & \cellcolor[HTML]{C0C0C0}7 & \cellcolor[HTML]{C0C0C0}8 & \cellcolor[HTML]{C0C0C0}9 & \cellcolor[HTML]{C0C0C0}10 & \cellcolor[HTML]{C0C0C0}11 & \cellcolor[HTML]{C0C0C0}12 & \cellcolor[HTML]{C0C0C0}13 & \cellcolor[HTML]{C0C0C0}14 & \cellcolor[HTML]{C0C0C0}15 & \cellcolor[HTML]{C0C0C0}16 \\ \hline
			\cellcolor[HTML]{C0C0C0}1 & \multicolumn{1}{c|}{0} & \multicolumn{1}{c|}{1} & \multicolumn{1}{c|}{1} & \multicolumn{1}{c|}{0} & \multicolumn{1}{c|}{1} & \multicolumn{1}{c|}{0} & \multicolumn{1}{c|}{1} & \multicolumn{1}{c|}{1} & \multicolumn{1}{c|}{1} & \multicolumn{1}{c|}{1} & \multicolumn{1}{c|}{2} & \multicolumn{1}{c|}{2} & \multicolumn{1}{c|}{3} & \multicolumn{1}{c|}{3} & \multicolumn{1}{c|}{4} & \multicolumn{1}{c|}{5} \\ \cline{2-17} 
			\cellcolor[HTML]{C0C0C0}2 & \multicolumn{1}{c|}{1} & \multicolumn{1}{c|}{1} & \multicolumn{1}{c|}{1} & \multicolumn{1}{c|}{1} & \multicolumn{1}{c|}{2} & \multicolumn{1}{c|}{2} & \multicolumn{1}{c|}{4} & \multicolumn{1}{c|}{5} & \multicolumn{1}{c|}{8} & \multicolumn{1}{c|}{11} & \multicolumn{1}{c|}{18} & \multicolumn{1}{c|}{25} & \multicolumn{1}{c|}{40} & \multicolumn{1}{c|}{58} & \multicolumn{1}{c|}{90} & \multicolumn{1}{c|}{135} \\ \cline{2-17} 
			\cellcolor[HTML]{C0C0C0}3, 4 & \multicolumn{1}{c|}{2} & \multicolumn{1}{c|}{1} & \multicolumn{1}{c|}{2} & \multicolumn{1}{c|}{3} & \multicolumn{1}{c|}{6} & \multicolumn{1}{c|}{9} & \multicolumn{1}{c|}{18} & \multicolumn{1}{c|}{30} & \multicolumn{1}{c|}{56} & \multicolumn{1}{c|}{99} & \multicolumn{1}{c|}{186} & \multicolumn{1}{c|}{335} & \multicolumn{1}{c|}{630} & \multicolumn{1}{c|}{1161} & \multicolumn{1}{c|}{2182} & \multicolumn{1}{c|}{4080} \\ \cline{2-17} 
			\cellcolor[HTML]{C0C0C0}5 & \multicolumn{1}{c|}{3} & \multicolumn{1}{c|}{2} & \multicolumn{1}{c|}{6} & \multicolumn{1}{c|}{13} & \multicolumn{12}{c|}{?} \\ \cline{2-17} 
			\cellcolor[HTML]{C0C0C0}6 & \multicolumn{1}{c|}{3} & \multicolumn{1}{c|}{2} & \multicolumn{1}{c|}{5} & \multicolumn{1}{c|}{10} & \multicolumn{1}{c|}{24} & \multicolumn{1}{c|}{50} & \multicolumn{1}{c|}{120} & \multicolumn{1}{c|}{270} & \multicolumn{1}{c|}{640} & \multicolumn{1}{c|}{1500} & \multicolumn{1}{c|}{3600} & \multicolumn{1}{c|}{8610} & \multicolumn{4}{c|}{$>10^4$} \\ \cline{2-17} 
		\end{tabular}
		\caption{\small{A table of motivic $\tilde{d}(w,N)$, which can be interpreted as number of "primitive constants" of weight $w$ and level $N$. The situation of level $5$ is more complicated as no simple formula is known their motivic dimensions.}}. 
		\label{primitiveconstants}
	\end{table}
	In the Mathematica package \textit{MultipleZetaValues} written by the author (see Appendix A), complete reduction of the following weights and levels are stored:
	\begin{itemize}
		\item Weight $\leq 14$ at level 1
		\item Weight $\leq 9$ at level 2
		\item Weight $\leq 5$ at level 3
		\item Weight $\leq 6$ at level 4
		\item plus some weights for higher levels. 
	\end{itemize}
	check Appendix A to see the chosen basis of $\widetilde{\textsf{CMZV}}^N_w$ that is used to store these results. Such compilation of values is important when we use them to evaluate certain definite integrals and infinite series. 
	
	\section{One dimensional definite integral}
	Before going into the full algorithm, we first do some examples. 
	
	\subsection{Some examples}
	\begin{example}
		We first do a baby example: evaluate $$I=\int_0^1 \frac{\Li_2(x) \log(1-x)}{x} dx$$
		using notations in Section 2.2, we have $\Li_2(x) = \int_0^x ab_0, \log(1-x) = -\int_0^x b_0$, so
		$$I = -\int_0^1 (\frac{1}{x}\int_0^x ab_0 \shuffle b_0) = -\int_0^1 a(ab_0 \shuffle b_0) = -\int_0^1 (2a^2b_0^2 + ab_0ab_0) 
		$$
		the integral can be converted into colored polylogarithm via (\ref{toitint}), since the level in this case is $1$, so they can be converted into multiple zeta function. We have $I = -2\zeta(3,1)-\zeta(2,2) = -\frac{\pi^4}{72}$.
	\end{example}
	
	\begin{example}
		Let level $N=2$, adopt notations in Section 2.2. Let $u$ be $a=\frac{dx}{x}$ or $b_0 = \frac{dx}{1-x}$ or $b_1 = \frac{-dx}{1+x}$, consider $$I = \int_0^1 \log^n x \log^m(1-x) \log^p(1+x) u$$
		Note that $$\log(1-x)^m = (-1)^m m! \int_0^x b_0^m \qquad \log(1+x)^p = (-1)^p p! \int_0^x b_1^p \qquad \log^n x = (-1)^n n! \int_x^1 a^n$$
		therefore
		$$I = (-1)^{n+m+p}m!n!p! \int_0^1 u\int_x^1 a^n \int_0^x b_0^m \shuffle b_1^p = (-1)^{n+m+p}m!n!p! \int_0^1 a^n u (b_0^m \shuffle b_1^p)$$
		the RHS can be converted directly into CMZV of level $2$. For example,
		$$\begin{aligned}\int_0^1 \frac{\log x \log(1-x) \log(1+x)}{x} dx &= -\int_0^1 a^2 (b_0\shuffle b_1) = -\int_0^1 a^2b_0b_1 + a^2b_1b_0 \\ 
			&= -\int_0^1 z_{3,0}z_{1,1} + z_{3,1}z_{1,0} \\
			&= -L_{3,1}(1,-1) -L_{3,1}(-1,-1) = -\zeta(3,\bar{1}) -\zeta(\bar{3},\bar{1})\\
			&= 2 \text{Li}_4\left(\frac{1}{2}\right)+\frac{7}{4} \zeta (3) \log (2)-\frac{3 \pi ^4}{160}+\frac{\log ^4(2)}{12}-\frac{1}{12} \pi ^2 \log ^2(2)
		\end{aligned}$$
		Integrals involving on $\log x, \log(1+x), \log(1-x)$ are subjected to intensive investigation in \cite{Au}, some information for weight $\leq 20$ are recorded there. 
	\end{example}
	
	\begin{example}
		This example illustrates the principal of regularization as indicated in Section 2.3. Consider the integral (see the Remark in Section 2.3)
		$$I = \int_0^1 \frac{\Li_3(x)}{1-x} dx$$
		Theorem \ref{regularization} enables us to do the following manipulation:
		$$I = \int_0^1 \frac{1}{1-x}\int_0^x a^2 b_0 = \int_0^1 b_0 a^2 b_0 = \text{\L}(b_0 a^2 b_0) = -\text{\L}(ab_0ab_0) - 2\text{\L}(a^2 b_0^2) = -\zeta(2,2)-2\zeta(3,1)$$
		so $I = -\frac{\pi^4}{72}$, it is the value of the following (convergent) integral:
		$$\int_0^1 \frac{\Li_3(x) - \zeta(3)}{1-x}dx$$
	\end{example}
	
	\subsection{Admissible rational function}
	We introduce the following definition. Let $N$ be a positive integer, we say that a rational function $R(x)\in \mathbb{C}(x)$ is \textit{$N$-admissible} if both $R(x)$ and $1-R(x)$ are of the form ($\mu = \exp(2\pi i /N)$)
	$$C x^d \prod_{i=0}^{N-1} (x-\mu^i)^{c_i} \qquad C\in \mathbb{C}\qquad  d,c_i\in \mathbb{Z}$$
	
	The following conjecture is very plausible, but author's limited knowledge leads no rigorous proof (For $N=1$ this is easy). 
	
	\begin{conjecture}
		For positive integer $N$, the number of $N$-admissible rational functions is finite.
	\end{conjecture}
	
	A large list of $4$-admissible rational function can be found in Appendix B. The list there \textit{has not been proved to be complete} by the author. \par 
	
	Recall the generalized polylogarithm defined in Example 2.5.
	
	\begin{theorem}\label{Liitint}  Let $\mathfrak{A} = \mathbb{C}\langle x_0,x_1\rangle$. For $N$-admissible $R(x)$, $w\in \{x_0,x_1\}^\ast$ of weight $n$\footnote{that is, number of $x_0$ plus number of $x_1$ equals $n$}, $c = R(0)$, we have the following iterated integral representation:
		\begin{equation}\tag{*}\Li_w(R(x)) = \sum_{k=1}^n c_k \int_0^x f_n \cdots f_k + \int_0^x f_n f_{n-1} \cdots f_1\end{equation}
		where $f_i = (R'/R) dx$ or $R'/(1-R) dx$ are differential forms, depending on whether the $i$-th letter (from the left) of $w$ is $x_0$ or $x_1$ respectively, and $c_k = \int_0^c f_k \cdots f_1$.\footnote{By principal of regularization in Section 2.3, if $c=\infty$, then invoke the asymptotic expansion of $\Li_w$ at infinity, ignore term of logarithmic growth keep just the constant term.}
		
		Therefore for any $N$-admissible $R(x)$, any $w\in \mathfrak{A}$, $\Li_w(R(x)) \in \int_0^x y $ for some $y\in \mathbb{C}\langle a,b_0,\cdots,b_{N-1}\rangle$. 
	\end{theorem}
	\begin{proof}
		The proof is not difficult. Proceed by induction, the case $n=1$ is evident. Note that (*) is true when $x=0$, thus it suffices to prove both sides are equal after differentiation, this is true by induction hypothesis.
	\end{proof}
	
	\begin{example}
		We compute $$I = \int_0^1 \frac{\Li_2(-\frac{1}{x}) \Li_2(\frac{4x}{(1+x)^2})}{x} dx$$
		since both $-1/x, 4x/(1+x)^2$ are $2$-admissible, let the level $N=2$, then in notation of Section 2.2, $a=dx/x, b_0 = dx/(1-x), b_1 = -dx/(1+x)$. 
		Because $$\Li_2(-1/x) = -\frac{\pi^2}{6} - \frac{1}{2}\log^2 x + o(1) \qquad x\to 0^+$$ using above theorem we have the (regularized) iterated integral (with $c_2 = -\pi^2/6, c_1 = 0, f_2 = -a, f_1 = a+b_1$):
		$$\Li_2(-1/x) = -\frac{\pi^2}{6} - \int_0^x a(a+b_1)$$
		similarly, $$\Li_2(\frac{4x}{(1+x)^2}) = \int_0^x (a+2b_1)(2b_0-2b_1)$$
		The shuffle product of $-\frac{\pi^2}{6} - a(a+b_1)$ and $(a+2b_1)(2b_0-2b_1)$ equals
		\scriptsize{ \begin{multline*}-\frac{1}{3} \pi ^2 ab_0+\frac{1}{3} \pi ^2 ab_1-6 a^3b_0+6 a^3b_1-4 a^2b_0a-4 a^2b_0b_1+4 a^2b_1a-8 a^2b_1b_0+12 a^2b_1b_1 -2 ab_0a^2\\-2 ab_0ab_1+2 ab_1a^2-6 ab_1ab_0+8 ab_1ab_1-4 ab_1b_0a-4 ab_1b_0b_1+4 ab_1b_1a-8 ab_1^2b_0+12 ab_1^3-4 b_1a^2b_0+4 b_1a^2b_1- 4 b_1ab_0a\\-4 b_1ab_0b_1+4 b_1ab_1a-4 b_1ab_1b_0+8 b_1ab_1^2-4 b_1b_0a^2-4 b_1b_0ab_1+4 b_1^2a^2+4 b_1^2ab_1-\frac{2}{3} \pi ^2 b_1b_0+\frac{2}{3} \pi ^2 b_1^2
		\end{multline*} }
		\normalsize{Concatenate this with $a = dx/x$ at the front, then convert it into $\text{\L}$, some terms requires (\ref{reg}) to convert to alternating multiple zeta function. Plug in all the level 2  CMZVs that pops up, we then obtain}
		\small{$$I = 24 \text{Li}_5\left(\frac{1}{2}\right)+16 \text{Li}_4\left(\frac{1}{2}\right) \log (2)+\frac{\pi ^2 \zeta (3)}{2}-\frac{93 \zeta (5)}{4}+\frac{7}{2} \zeta (3) \log ^2(2)+\frac{7 \log ^5(2)}{15}-\frac{1}{3} \pi ^2 \log ^3(2)-\frac{41}{360} \pi ^4 \log (2)$$}
	\end{example}
	~\\[0.02in]
	
	For positive integer $N$, let 
	$$\mathcal{C}^N = \{R(0) | R(x)\text{ is }N\text{-admissible}\}$$
	then it is not difficult to show $\mathcal{C}^1 = \{0,1,\infty\}$. $\mathcal{C}^N$ for other $N$ is not known with certainty to the author, but very likely they are
	\begin{equation}\tag{**}\mathcal{C}^2 \overset{?}{=} \{0,1,\frac{1}{2},2,\infty\} \qquad \mathcal{C}^4 \overset{?}{=} \{0,1,\frac{1}{2},2,\pm i, 1\pm i, \frac{1\pm i}{2}, \infty\}\end{equation}
	
	Although the set of $N$-admissible rational function remains elusive, the author has faith in the following conjecture:
	\begin{conjecture}
		Let $w\in \{x_0,x_1\}^\ast$ be a weight $n$ word, then for any $x\in \mathcal{C}^N$, $\Li_w(x) \in \textsf{CMZV}^N_w$. 
	\end{conjecture}
	
	The conjecture is known for any word $w$ and values of $x$ in (**). Its intractability mainly comes from the possible incompleteness of values listed in (**). Assuming this conjecture, we have the main result of this paper:
	
	\begin{theorem}\label{intmaintheo}
		Let $R_i(x)$ be $N$-admissible rational functions, $f(x) = 1/(x-d)$, with $d=0$ or an $N$-th root of unity, $w_i\in \{x_0,x_1\}^\ast$, then
		$$\int_0^1 f(x) \prod_i \Li_{w_i}(R_i(x)) dx \in \textsf{CMZV}^N_{1+\sum |w_i|}$$
		with $|w_i|$ the weight of $w_i$. 
	\end{theorem}
	\begin{proof}
		All aspects of the algorithm has been illustrated in Example 3.6.
	\end{proof}
	
	\begin{remark}
		Since $\Li_w$ is multi-valued, and $R_i(x)$ might wind around branch point more than one complete cycle as $x$ varies from $0$ to $1$, the above assertion is best interpreted by \textit{thinking of} $\Li_w(R(x))$ as RHS of $(*)$. If $R_i$ all satisfy $R_i([0,1])\subset [0,1]$, then such issue do not arise. 
	\end{remark}
	
	\section{Infinite sums that are reducible to CMZVs}
	\subsection{Apéry-like series involving harmonic numbers}
	Apart from classical Euler sums (for example, in \cite{flajolet}), which are directly reducible to CMZV, we state a few more which follows nicely from our main theorem Theorem \ref{intmaintheo}. Recall the harmonic number $H_n(w)$ defined in Example 2.2. 
	\begin{theorem}
		For any word $w\in \mathfrak{A}_1$\footnote{i.e. algebra generated by words ending in $x_1$}, positive integer $s$ 
		$$\sum_{n=1}^\infty \frac{H_{n-1}(w)}{n^s} \left[ {{4^n}{\binom{2n}{n}^{ - 1}}} \right] \in \textsf{CMZV}^2_{s+|w|} \qquad s>1$$
		$$\sum_{n=1}^\infty \frac{H_{n-1}(w)}{n^s} \left[ {{2^n}{\binom{2n}{n}^{ - 1}}} \right] \in \textsf{CMZV}^4_{s+|w|}$$
	\end{theorem}
	\begin{proof}
		Consider $$\int_0^1 \frac{1}{x}\Li_w(R(x))dx \qquad R(x) = \frac{4x}{(1+x)^2} \text{ or } \frac{2x}{(1+x)^2}$$
		the first $R(x)$ is $2$-admissible, the second $R(x)$ is $4$-admissible. Using series expansion of $\Li_w(x)$ (see Example 2.5), integrate termwise, and use 
		$$\int_0^1 \frac{x^{n-1}}{(1+x)^{2n}} dx = \frac{1}{n}\binom{2n}{n}^{-1}$$
		completes the proof.
	\end{proof}
	
	\begin{example}
		$$\begin{aligned}
			\sum\limits_{n = 2}^\infty  {\frac{{H_{n - 1}^{(2)}}}{{{n^3}}}\left[ {{4^n}{\binom{2n}{n}^{ - 1}}} \right]}  &= \frac{1}{{12}}{\pi ^4}\log 2 - \frac{3}{4}{\pi ^2}\zeta (3) + \frac{{31}}{8}\zeta (5) \\
			\sum\limits_{n = 2}^\infty  {\frac{{H_{n - 1}}}{{{n^4}}}\left[ {{4^n}{\binom{2n}{n}^{ - 1}}} \right]}  &= 32 \text{Li}_5\left(\frac{1}{2}\right)-\frac{\pi ^2 \zeta (3)}{2}-\frac{155 \zeta (5)}{8}-\frac{1}{15} 4 \log ^5(2)+\frac{4}{9} \pi ^2 \log ^3(2)+\frac{23}{180} \pi ^4 \log (2) \\
			\sum\limits_{n = 2}^\infty  {\frac{{H_{n - 1}^{(3)}}}{{{n^2}}}\left[ {{4^n}{\binom{2n}{n}^{ - 1}}} \right]}  &= \frac{93 \zeta (5)}{4}-\frac{7 \pi ^2 \zeta (3)}{4} \qquad \qquad 
			\sum\limits_{n = 2}^\infty  {\frac{{H_{n - 1}^{(2)}}}{{{n^2}}}\left[ {{2^n}{\binom{2n}{n}^{ - 1}}} \right]} = \frac{\pi ^4}{384} \\
			\sum\limits_{n = 2}^\infty  {\frac{{H_{n - 1}^{(2)}}}{{{n^3}}}\left[ {{2^n}{\binom{2n}{n}^{ - 1}}} \right]} &= \frac{\pi ^3 C}{24}-\pi  \beta(4)-\frac{3 \pi ^2 \zeta (3)}{128}+\frac{527 \zeta (5)}{256}+\frac{1}{384} \pi ^4 \log (2) \\
			\sum\limits_{n = 2}^\infty  {\frac{{H_{n - 1}^{(3)}}}{{{n^2}}}\left[ {{2^n}{\binom{2n}{n}^{ - 1}}} \right]} &= -3 \pi  \beta(4)-\frac{35 \pi ^2 \zeta (3)}{128}+\frac{1581 \zeta (5)}{128}
		\end{aligned}$$
		here $C$ is the Catalan constant, $\beta$ is Dirichlet beta function. 
	\end{example}
	
	In order to derive more interesting examples, we need an integral operator that gives out generalized harmonic numbers $H_n(w)$. Recall our notation of a variation of iterated integral defined in last part of Section 2.1, for positive integer $s$, set
	$$\chi_s = \begin{cases}\overline{x_1} &\qquad \text{ if }s=1 \\  
		\overline{x_0}x_1 &\qquad \text{ if }s=2 \\
		\overline{x_0}x_0^{s-2}x_1 &\qquad \text{ if }s\geq 3
	\end{cases}$$
	Writing $x_0 = dx/x, x_1 = dx/(1-x)$, for positive integers $s_i$, define the linear operator $D_{s_1,\cdots,s_k}$:
	$$D_{s_1,\cdots,s_k} f = \int_0^1 x_0^{s_1-1}x_1\chi_{s_2}\cdots \chi_{s_k} \int_x^1 f(x) dx $$
	
	\begin{lemma}
		For positive integer $n$, write $w = x_0^{s_k-1}x_1\cdots x_0^{s_1-1}x_1$
		\begin{equation}\tag{**}D_{s_1,\cdots, s_k} (nx^{n-1}) = H_n^\star (w) \end{equation}
		where (note that the order of $s_1,\cdots,s_k$ in $w$ has been reversed), $$H_n^\star(w) = \sum_{n\geq n_k\geq \cdots\geq n_1\geq 1} \frac{1}{n_1^{s_k}\cdots n_k^{s_1}}$$
	\end{lemma}
	\begin{proof}
		Proceed by induction on $k$, the case for $k=1$ is evident:
		$$D_{s}(nx^{n-1}) = \int_0^1 \frac{1}{x_1}\int_0^{x_1} \frac{1}{x_2} \cdots \int_0^{x_{s-1}} \frac{1-x_s^n}{1-x_s} dx_s = \sum_{i=1}^n \frac{1}{i^s}$$
		Now one easily computes, via induction hypothesis:
		$$D_{s_1,\cdots, s_k} (nx^{n-1} - (n-1) x^{n-2}) = \frac{1}{n^{s_k}} \sum_{n\geq n_{k-1}\geq \cdots\geq n_1\geq 1} \frac{1}{n_1^{s_{k-1}}\cdots n_{k-1}^{s_1}} \quad n\geq 1$$
		Because $(**)$ is true when $n=0$, the above forward difference easily implies the result. 
	\end{proof}
	
	Now consider a holomorphic function $f(z)$ on $|z|<1$ defined by the power series:
	$f(z) = \sum_{n=1}^\infty a_n z^n$. We shall assume $0<z<1$, so ($x_0 = dz/z$):
	$$\sum_{n=1}^\infty \frac{a_n}{n^{s-1}}z^n  = \int_0^z x_0^{s-2} \frac{f(z) dz}{z}$$
	taking the operator $D_{s_1,\cdots,s_k}$ on both sides (this is legitimate via dominated convergence theorem, as long as LHS converges), we have
	$$\sum_{n=1}^\infty \frac{a_n}{n^s} H_n^\star(w) = D_{s_1,\cdots,s_k} \left(\frac{1}{z}\int_0^z x_0^{s-2} \frac{f(z) dz}{z}\right)\qquad w = x_0^{s_k-1}x_1\cdots x_0^{s_1-1}x_1$$
	If $s=1$, then term inside the parenthesis should be interpreted as $\frac{f(z)}{z}$. Therefore, by switching $\int_x^1$ that occurs in $\chi_s$ back to $\int_0^1$, we see that $\sum_{n=1}^\infty \frac{a_n}{n^s} H_n(w)$ is always a $\mathbb{Z}$-linear combination of terms of form $$\left(\int_0^1 \omega_1\cdots \omega_{i_1}\right)\cdots \left(\int_0^1 \omega_{i_{k-1}+1}\cdots \omega_{i_k}\right)$$
	with $i_{k} = s+s_1+\cdots+s_k = |w|+s$. ($|w|$ is the weight of word $w$) The last differential form $\omega_{|w|+s}$ is $f(z)/z dz$, and all previous $\omega_i$ are either $dz/z$ or $dz/(1-z)$. 
	
	\begin{theorem}
		For any word $w\in \mathfrak{A}_1$, positive integer $s$, coprime integers $p<0, q>0$, $p+q\geq 1$. We have $$\sum_{n=1}^\infty \frac{H_{n-1}(w)}{n^s} (-1)^n \binom{p/q}{n} \in \textsf{CMZV}^q_{s+|w|}$$
	\end{theorem}
	\begin{proof}
		Since $H_{n-1}(w)$ can be written as a linear combination of $n^{-m_i}H_n^\star(w_i)$, with $m_i+|w_i| = |w|$. It suffices to the assertion for $H_{n}^\star(w)$. By our observation above, the series is a $\mathbb{Z}$-linear combination of $$\left(\int_0^1 \omega_1\cdots \omega_{i_1}\right)\cdots \left(\int_0^1 \omega_{i_{k-1}+1}\cdots \omega_{i_k}\right)$$
		with $i_k = s+|w|, \omega_{|w|+s} = ((1-x)^{p/q}-1)/x dx$, and all other $\omega_i$ are either $dx/x$ or $dx/(1-x)$. All terms except the last one in the above displayed equation are level $1$ CMZVs. For the last iterated integral involving $\omega_{|w|+s}$, pull it back by $g: x\to 1-x^q$ using (\ref{itintpullback}), then the path of integration is still $[0,1]$ (direction reversed), and $$g^\ast \frac{dx}{x} = \frac{-qx^{q-1}}{1-x^q}dx \qquad g^\ast \frac{dx}{1-x} = \frac{-q}{x} dx\qquad  g^\ast \omega_{|w|+s} = \frac{qx^{q-1}-qx^{p+q-1}}{1-x^q}dx$$
		since $0\leq p+q-1< q$, all above differential forms can be converted into linear combination of $dx/(1-\exp(2\pi i k /q)), k=0,1,\cdots,q-1$, so the last iterated integral can be converted into level $q$ CMZV, the completes the proof.
	\end{proof}
	
	\begin{corollary}
		For any word $w\in \mathfrak{A}_1$, positive integer $s$ 
		$$\sum_{n=1}^\infty \frac{H_{n-1}(w)}{n^s} \left[ 4^{-n} \binom{2n}{n} \right] \in \textsf{CMZV}^2_{s+|w|}$$
	\end{corollary}
	\begin{proof}
		Apply above theorem to $p=-1, q=2$ and use $\binom{-1/2}{n}(-1)^n = 4^{-n}\binom{2n}{n}$. 
	\end{proof}
	
	\begin{example}
		$$\begin{aligned}\sum_{n=2}^\infty \frac{H_{n-1}}{n^3} \left[ 4^{-n} \binom{2n}{n} \right] &= 8 \text{Li}_4\left(\frac{1}{2}\right)+2 \zeta (3) \log (2)-\frac{11 \pi ^4}{180}+\log ^4(2) \\
			\sum_{n=2}^\infty \frac{H_{n-1}^{(2)}}{n^2} \left[ 4^{-n} \binom{2n}{n} \right] &= \zeta (3) \log (2)+\frac{\pi ^4}{120}+\frac{2 \log ^4(2)}{3}-\frac{1}{3} \pi ^2 \log ^2(2) \\
			\sum_{n=2}^\infty \frac{H_{n-1}^{(3)}}{n} \left[ 4^{-n} \binom{2n}{n} \right] &= -8 \text{Li}_4\left(\frac{1}{2}\right)-3 \zeta (3) \log (2)+\frac{7 \pi ^4}{90}+\frac{\log ^4(2)}{3}
		\end{aligned}$$
	\end{example}
	
	Corollary 4.5 was already obtained by Wang and Xu in \cite{wang2021alternating}, a Maple package to calculate such series was also written. More computational approaches can be found in \cite{kalmykovBinomial}. First half of Theorem 4.1 and (a weaker version of) Corollary 4.5 (in terms of ordinary harmonic numbers) are recently and independently proved by Zhao \cite{ZhaoMingHao2}, the idea of iterated integral is already germinating in this paper. 
	
	\subsection{Fourier-Legendre expansion of generalized polylogarithm}
	Let $P_n(x) = \frac{1}{2^nn!}\frac{d^n}{dx^n}(x^2-1)^n$ be the classical Legendre polynomials, we will focus on the shifted version $\tilde{P}_n(x) = P_n(2x-1)$. $\{\tilde{P}_n(x)\}$ forms a complete orthogonal family on $L^2(0,1)$, with
	$$\int_0^1 {{{\widetilde P}_n}(x){{\widetilde P}_m}(x)dx}  = \frac{{{\delta _{mn}}}}{{2n + 1}}$$
	For $f\in L^2(0,1)$, we will write $f\sim \sum\limits_{n \ge 0} {{c_n}{{\widetilde P}_n}(x)} $ to represent the expansion of $f$ in terms of $\tilde{P}_n(x)$. The expansion converges to $f(x)$ in $L^2$ norm (\cite{andrews1999special}), we don't need results about pointwise convergence.
	
	\begin{proposition}
		Let $f\in L^2(0,1), f(x)\sim\sum\limits_{n \ge 0} {{c_n}{{\widetilde P}_n}(x)}$. Then $$f(1 - x)\sim \sum\limits_{n \ge 0} {{{( - 1)}^n}{c_n}{{\widetilde P}_n}(x)} \qquad \int_0^x {f(x)dx} \sim \sum\limits_{n \ge 0} {\left[ {\frac{{{c_{n - 1}}}}{{2(2n - 1)}} - \frac{{{c_{n + 1}}}}{{2(2n + 3)}}} \right]{{\widetilde P}_n}(x)} $$
		If $f(x)/(1-x) \in L^2(0,1)$, then $$\frac{{f(x)}}{{1 - x}}\sim \sum\limits_{n \ge 0} {(2n + 1)\left( {\int_0^1 {\frac{{f(x)}}{{1 - x}}dx}  - 2\sum\limits_{m = 1}^n {\frac{1}{m}\sum\limits_{k = 0}^{m - 1} {{c_k}} } } \right)} {\widetilde P_n}(x)$$
		If $f(x)/x \in L^2(0,1)$, then $$\frac{{f(x)}}{x}\sim \sum\limits_{n \ge 0} {{{( - 1)}^n}(2n + 1)\left( {\int_0^1 {\frac{{f(x)}}{x}dx}  - 2\sum\limits_{m = 1}^n {\frac{1}{m}\sum\limits_{k = 0}^{m - 1} {{{( - 1)}^k}{c_k}} } } \right)} {\widetilde P_n}(x)$$
	\end{proposition}
	\begin{proof}
		The expansion about $f(1-x)$ follows from $P_n(x) = (-1)^n P_n(-x)$, that of $\int_0^x f(x) dx$ follows from $\int {{{\widetilde P}_n}(x)dx}  = \frac{{{{\widetilde P}_{n + 1}}(x) - {{\widetilde P}_{n - 1}}(x)}}{{2(2n + 1)}}$. The last assertion about $f(x)/x$ follows from that of $f(x)/(1-x)$. Therefore it remains to prove the expansion of $f(x)/(1-x)$, we include a quick proof due to lack of reference. Let $F_{n+1}(x) = (n+1)[\tilde{P}_{n+1}(x)-\tilde{P}_n(x)]$, then the three-term recurrence of $\tilde{P}_n$ implies
		$$\int_0^1 \frac{f(x)}{1-x}(F_{n+1}-F_n(x)) dx = -2c_n$$
		telescoping, after that divide by $n+1$, another telescoping gives the result.
	\end{proof}
	
	Recall our notation about Hoffman-Racinet algebra $z_{k,i} = a^{k-1}b_i$. For $w = z_{k_1,i_1}\cdots z_{k_2,i_2}\cdots \in \mathfrak{A}^2_1$, define an $\mathbb{C}$-linear map $\theta_i: \mathfrak{A}^2_1 \to \mathfrak{A}^2_1$ ($i=0, 1$) by $\theta_i(w) = {z_{{k_1} + 1,{i_1} + i}}{z_{{k_2},{i_2}}}\cdots$. 
	\begin{proposition}
		Let $w\in \mathfrak{A}^2_1$, $f(x)\in L^2(0,1)$. If
		$f(x)\sim c + \sum\limits_{n \ge 1} {(2n + 1) H_n^\star (w)} {\widetilde P_n}(x)$, 
		then $$\frac{1}{x}\int_0^x {f(x)dx} \sim \sum\limits_{n \ge 0} {(2n + 1){{( - 1)}^n}H_n^\star (C - {\theta _1}w)} {\widetilde P_n}(x) \qquad C = \int_0^1 \frac{1}{x}\int_0^x {f(t)dt}$$
		$$\frac{1}{{1 - x}}\int_x^1 {f(x)dx} \sim \sum\limits_{n \ge 0} {(2n + 1)H_n^\star (C + {\theta _0}w - 2{b_0}w){{\widetilde P}_n}(x)} \qquad C = \int_0^1 \frac{1}{1-x}\int_x^1 {f(t)dt}$$
		If
		$f(x)\sim c + \sum\limits_{n \ge 1} {(2n + 1) (-1)^n H_n^\star (w)} {\widetilde P_n}(x)$, 
		then $$\frac{1}{x}\int_0^x {f(x)dx} \sim \sum\limits_{n \ge 0} {(2n + 1){{( - 1)}^n}H_n^\star ({C + {\theta _0}w - 2{b_0}w})} {\widetilde P_n}(x) \qquad C = \int_0^1 \frac{1}{x}\int_0^x {f(t)dt}$$
		$$\frac{1}{{1 - x}}\int_x^1 {f(x)dx} \sim \sum\limits_{n \ge 0} {(2n + 1)H_n^\star ({C - {\theta _1}w}){{\widetilde P}_n}(x)} \qquad C = \int_0^1 \frac{1}{1-x}\int_x^1 {f(t)dt}$$
	\end{proposition}
	\begin{proof}
		Immediately follows from the previous proposition.
	\end{proof}
	
	\begin{proposition}
		Let $w\in \{x_0,x_1\}^\ast$ that contains $x_1$. Then there exists an $w_0\in \mathfrak{A}^2_1$ so that
		$$\frac{\Li_w(x)}{x} \sim \sum_{n\geq 0}(2n+1) (-1)^n H_n^\ast(w_0) \tilde{P}_n(x)$$
	\end{proposition}
	\begin{proof}
		Use induction on weight of $w$. For $w=x_1$, we have
		$$\frac{-\log(1-x)}{x} dx \sim \sum_{n\geq 0}(2n+1) (-1)^n \left[\frac{\pi^2}{6}+2\sum_{k=1}^n \frac{(-1)^k}{k^2}\right] \tilde{P}_n(x)$$
		For $w$ of weight $n$, applying above proposition repeatedly shows that $\frac{1}{x}\int_0^x u v_1 \cdots v_{n-1}$, with $u = x_0 = dx/x$ or $x_1 = dx/(1-x)$, $v=x_0,x_1,\overline{x_0}$ or $\overline{x_1}$ (see last paragraph of Section 2.1) has FL-expansion of desired form. Convert $\int_x^1$ into $\int_0^1 - \int_0^x$, induction hypothesis completes the proof.
	\end{proof}
	
	\begin{example}
		For each $w$ below, we give\footnote{We abbreviate $\sum_{n\geq n_1>n_2>n_3\geq 1}\frac{(-1)^{n_3}}{n_1 n_2 n_3^2}$ as $H_{1,1,-2}$, and similarly for other harmonic numbers.} $c_n$ in $\Li_w(x)/x \sim \sum_{n\geq 0} (2n+1) (-1)^n c_n \tilde{P}_n(x)$: 
		$$w = x_0^2 x_1\qquad c_n = \frac{\pi^4}{90} + \frac{\pi^2}{3}H_2 + 2H_{-4}+\frac{2\pi^2}{3}H_{1,1}+4H_{1,-3}+4H_{2,-2}+8H_{1,1,-2}+\zeta(3)H_1$$
		$$w = x_0^2x_1x_0\qquad c_n = \frac{\pi^4}{15}H_1 + 2H_5+4H_{1,4}+4H_{2,3}+8H_{1,1,3}-4\zeta(3)H_{2}-8\zeta(3)H_{1,1}-4\zeta(5)$$
		\begin{multline*}w=x_0^3x_1\qquad c_n =  -\frac{\pi^4}{45}H_1-\frac{\pi^2}{3}H_3-2H_{-5}-\frac{2\pi^2}{3}H_{1,2}-4H_{1,-4}-\frac{2\pi^2}{3}H_{2,1}-4H_{2,-3}-4H_{3,-2}-\frac{4\pi^2}{3}H_{1,1,1}\\ -8H_{1,1,-3}-8H_{1,2,-2}-8H_{1,1,-1}-16H_{1,1,1,-2}+2\zeta(3)H_2+4\zeta(3)H_{1,1}+\zeta(5)\end{multline*}
		Note that last entry records the Fourier-Legendre expansion of $\Li_4(x)/x$. 
	\end{example}
	
	Our main goal of introducing Legendre polynomial is to prove the following theorem. Let $K(x) = \frac{\pi}{2} {_2F_1}(\frac{1}{2},\frac{1}{2};1;x)$ be the complete elliptic integral of first kind. 
	\begin{theorem}
		Let $w\in \{x_0,x_1\}^\ast$ that contains $x_1$ of weight $|w|$, then (both integrals converge)
		$$\int_0^1 \frac{K(x) \Li_w(x)}{x} dx, \int_0^1 \frac{K(1-x)\Li_w(x)}{x}dx  \in \textsf{CMZV}^4_{|w|+2}$$
	\end{theorem}
	\begin{proof}
		The proof uses $K(x) \sim \sum_{n\geq 0}\frac{2}{2n+1} \tilde{P}_n(x)$ (see \cite{cohl2012generalizations}). Using expansion of $\Li_w(x)/x$ obtained above, the first integral becomes a linear combination of sums of form $\sum_{n\geq 0} (-1)^n H_n(w)/(2n+1)$, with $w\in \mathfrak{A}^2_1$ in level $2$ Hoffman-Racinet algebra. This is a convergent Euler sum, and can be converted into level $4$ CMZV. This completes the proof for $\int_0^1 \frac{K(x) \Li_w(x)}{x} dx$. For the second one, it reduces into a linear combination of $\sum_{n\geq 0} H_n(w)/(2n+1)$, which diverges. However, its regularized value ($c$ in the statement of Proposition 2.3) is still a level 4 CMZV. Although individual sums might diverge, the overall sum $\sum \frac{c_n}{2n+1}$ (see notations of Examples above) must converge, hence it is legitimate to replace divergent sum by its regularized value, completing the proof.
	\end{proof}
	
	\begin{example} For example, 
		\small $$\begin{aligned}\int_0^1 \frac{K(x)\Li_2(x)}{x}dx &= -\frac{2 \pi ^2 C}{3}-512 L_4-128 \log (2) L_3+400 \beta(4)+\frac{4}{3} \pi  \log ^3(2)-3 \pi ^3 \log (2) \\ 
			\int_0^1 \frac{K(1-x)\Li_2(x)}{x}dx &= -32 \pi  L_3+32 \text{Li}_4\left(\frac{1}{2}\right)+\frac{41 \pi ^4}{90}+\frac{4 \log ^4(2)}{3}-\frac{1}{3} \pi ^2 \log ^2(2)\\
			\int_0^1 \frac{K(1-x)\log x \log(1-x)}{x} dx&= 512 L_4+128 \log (2) L_3-416 \beta(4)+7 \pi  \zeta (3)-\frac{4}{3} \pi  \log ^3(2)+3 \pi ^3 \log (2)
		\end{aligned}$$
		\normalsize with $L_n = \Im\left(\text{Li}_n\left(\frac{1}{2}+\frac{i}{2}\right)\right)$. 
	\end{example}
	
	\begin{remark}
		Let $w$ be any word in $\{x_0,x_1\}^\ast$. Since $K(0)=\pi/2, K(x) = 2\log 2 -\log(1-x)/2 + o(1)$ as $x\to 1^-$, there exists positive integers $s_i$ and $a_i,c\in \mathbb{C}$ such that
		$$\int_\alpha^1 \frac{f(x) \Li_w(x)}{g(x)} dx = \sum a_i \log^{s_i} \alpha + c + o(1) \qquad \alpha\to 0^+$$
		with $f(x) = K(x)$ or $K(1-x)$, $g(x) = x$ or $1-x$. $c$ will be defined as the regularized value of the integral. Using machinery developed, it is not difficult to prove that such regularized value is also in $\textsf{CMZV}^4$. 
	\end{remark}
	
	Our inspiration to work with Legendre polynomials was sparked from \cite{campbell2019interplay}. However, the author still hopes to find a way to prove Theorem 4.11 via integral transformations only, this has two advantages. Firstly, we already have a well-developed regularization theory for shuffle CMZV (Theorem 2.4); but when working with Fourier-Legendre expansion, one has to take care of convergence everywhere due to lack of suitable regularization theory. Secondly, there are cases not included in Theorem 4.11 but nonetheless yield closed forms in CMZVs, see \cite{chan2013legendre}, it is hoped that by figuring out the integral transformations, we can differentiate whether a certain hypergeometric-type integral is related to CMZVs. \par
	Fourier-Legendre expansion utilized to series evaluations can also found in \cite{levrie2010using}.
	
	\subsection{Series involving binomial coefficient squared and harmonic numbers}
	We elucidate the nature of the following intractable sums:
	\begin{theorem}
		For any word $w\in \mathfrak{A}_1$, positive integer $s$
		$$\sum_{n=1}^\infty \frac{H_{n-1}(w)}{n^s} \left[ 4^{-n} \binom{2n}{n} \right]^2 \in \frac{1}{\pi}\textsf{CMZV}^4_{s+|w|+1}$$
	\end{theorem}
	\begin{proof}
		It suffices to prove the assertion for $H_n^\star(w)$. Set $a_n = \left[ 4^{-n} \binom{2n}{n} \right]^2$, by our discussion preceding Theorem 4.4, we have $\sum_{n=1}^\infty a_n x^n = \frac{2}{\pi}K(x) -1$, so
		$$\sum_{n=1}^\infty \frac{H_n^\star(w)}{n^s} a_n = \frac{1}{\pi} D_{s_1,\cdots,s_k} \left(\frac{1}{x}\int_0^x x_0^{s-2} \frac{(2K(x)-\pi) dx}{x}\right)\qquad w = x_0^{s_k-1}x_1\cdots x_0^{s_1-1}x_1$$
		Apart from the factor $1/\pi$, RHS is a $\mathbb{Z}$-linear combination of $$\left(\int_0^1 \omega_1\cdots \omega_{i_1}\right)\cdots \left(\int_0^1 \omega_{i_{k-1}+1}\cdots \omega_{i_k}\right)$$
		with $i_k = s+|w|, \omega_{|w|+s} = (2K(x)-1)/x dx$, and all other $\omega_i$ are either $dx/x$ or $dx/(1-x)$. All terms except the last one in the above displayed equation are level $1$ CMZVs. For the last iterated integral involving $\omega_{|w|+s}$, pull it back by $x\to 1-x$, gives, for some $w'\in \mathfrak{A}_1$, $$\int_0^1 \frac{(2K(1-x)-\pi) \Li_{w'}(x)}{1-x} dx = \int_0^1 \frac{(2K(x)-\pi)\Li_{w'}(1-x)}{x}dx $$
		Since $\Li_{w'}(1-x)$ can be written as $\sum c_i \Li_{w_i}(x)$ for some other words $w_i$ and $c_i\in \mathbb{R}$, weight preserved, Theorem 4.11 says the above (regularized) integral is a level 4 CMZV. Completing the proof.
	\end{proof}
	
	\begin{theorem}
		For any word $w\in \mathfrak{A}_1$, positive integer $s\geq 3$,\footnote{This condition is imposed to ensure convergence of the series.}
		$$\sum_{n=1}^\infty \frac{H_{n-1}(w)}{n^s} \left[ 4^{-n} \binom{2n}{n} \right]^{-2} \in \textsf{CMZV}^4_{s+|w|}$$
	\end{theorem}
	\begin{proof}
		Denote the series by $S$. Set $v=x_0^{s-3}x_1w$, then 
		$$S = \iint_{(0,1)^2} \frac{\Li_v(16xy(1-x)(1-y))}{xy} dxdy = \iint_{(0,1/2)^2} \frac{\Li_v(16xy(1-x)(1-y))}{x(1-x)y(1-y)} dxdy$$
		replace $x$ by $(1-\sqrt{1-x})/2$, and $y$ by $(1-\sqrt{1-y})/2$, we have $$S = \iint_{{{(0,1)}^2}} {\frac{{{{{\mathop{\rm Li}\nolimits} }_v}(xy)}}{{xy\sqrt {1 - x} \sqrt {1 - y} }}dxdy} = \iint_{0 < x < y < 1} {\frac{{{{{\mathop{\rm Li}\nolimits} }_v}(x)}}{{x\sqrt y \sqrt {y - x} \sqrt {1 - y} }}dxdy} $$
		integrating respect to $y$ gives $S =2 \int_0^1 {\frac{{K(1 - x){{{\mathop{\rm Li}\nolimits} }_v}(x)}}{x}dx}$, which is in $\textsf{CMZV}^4_{s+|w|}$.
	\end{proof}
	
	\begin{example}
		\small
		$$\begin{aligned}\sum_{n=1}^\infty \frac{1}{n^4} \left[ 4^{-n} \binom{2n}{n} \right]^{-2}&= -64 \pi  L_3+64 \text{Li}_4\left(\frac{1}{2}\right)+\frac{41 \pi ^4}{45}+\frac{8 \log ^4(2)}{3}-\frac{2}{3} \pi ^2 \log ^2(2) \\
			\sum_{n=1}^\infty \frac{H_{n-1}}{n^3} \left[ 4^{-n} \binom{2n}{n} \right]^{-2} &= 32 C^2-32 \pi  C \log (2)-64 \pi  L_3+\frac{3 \pi ^4}{2}+2 \pi ^2 \log ^2(2) \\
			\sum_{n=1}^\infty \frac{H_n}{n^2} \left[ 4^{-n} \binom{2n}{n} \right]^{2}& = \frac{1024 L_4}{\pi }+\frac{256 \log (2) L_3}{\pi }-\frac{800 \beta(4)}{\pi }+\zeta (3)-\frac{8}{3} \log ^3(2)+\frac{20}{3} \pi ^2 \log (2) \\
			\sum_{n=1}^\infty \frac{H_n^{(2)}}{n} \left[ 4^{-n} \binom{2n}{n} \right]^{2}& = -\frac{4 \pi  C}{3}-\frac{32 \beta(4)}{\pi }+12 \zeta (3) 
		\end{aligned}$$
		\normalsize with $L_n = \Im\left(\text{Li}_n\left(\frac{1}{2}+\frac{i}{2}\right)\right)$. 
	\end{example}
	
	Series involving square central binomial coefficients have been intensively studied (via ingenious, elementary means) in \cite{campbell2019new}, \cite{campbell2019interplay}, \cite{cantarini2019interplay}, \cite{campbell2020hypergeometry}, \cite{campbell2018series} and \cite{campbell2017integral}. A lots of closed-forms in these papers are essentially CMZVs. Virtually all of them use Fourier-Legendre techniques. 
	
	\subsection{Rapidly converging series}
	For $x\in \mathbb{C}, x\notin (0,1)$, denote $\omega(c) = \frac{dx}{x-c}$.
	
	\begin{proposition}
		$$\int_0^1 \omega(c_1)\cdots \omega(c_n) = (-1)^n \int_0^1 \omega(1-c_n)\cdots \omega(1-c_1)$$
	\end{proposition}
	\begin{proof}
		Immediate from first rule of Proposition 2.1 and (\ref{itintpullback})
	\end{proof}
	
	\begin{theorem}
		For positive integer $s$,
		$$\sum_{n=1}^\infty \frac{(-1)^n}{n^s 2^n \binom{2n}{n}} \in \textsf{CMZV}^2_s$$
	\end{theorem}
	\begin{proof}
		The summation equals to $$I = \int_0^1 \frac{1}{x}\Li_{s-1}\left(-\frac{1}{2}x(1-x)\right) dx$$
		by expanding $\Li_{s-1}$ and termwise integration. Using method as in (\ref{Liitint}), we have
		$$\Li_{s-1}\left(-\frac{1}{2}x(1-x)\right) = \int_0^x (\omega(0)+\omega(1))^{s-2} (-\omega(-1) - \omega(2))$$
		so $I$ will be the iterated integral of a combination of $\omega(c_1)\cdots \omega(c_s)$, with $c_i \in \{0,1,-1,2\}$ with $-1,2$ \textit{never both occur in the same word}. If a word contains only $\omega(0),\omega(1),\omega(-1)$ then it is already a level 2 CMZV, if a word contains only $\omega(0),\omega(1),\omega(2)$, then above proposition transforms it to a level 2 CMZV, so $I \in \textsf{CMZV}^2_s$. 
	\end{proof}
	
	\begin{example}
		$$\begin{aligned}
			\sum_{n=1}^\infty \frac{(-1)^n}{n^3 2^n \binom{2n}{n}} &= \frac{\log ^3(2)}{6}-\frac{\zeta (3)}{4} \\ 
			\sum_{n=1}^\infty \frac{(-1)^n}{n^4 2^n \binom{2n}{n}} &= -4 \text{Li}_4\left(\frac{1}{2}\right)-\frac{13}{4} \zeta (3) \log (2)+\frac{7 \pi ^4}{180}-\frac{1}{24} 5 \log ^4(2)+\frac{1}{6} \pi ^2 \log ^2(2) \\
			\sum_{n=1}^\infty \frac{(-1)^n}{n^5 2^n \binom{2n}{n}} &= -10 \text{Li}_5\left(\frac{1}{2}\right)-6 \text{Li}_4\left(\frac{1}{2}\right) \log (2)+\frac{19 \zeta (5)}{2}-\zeta (3) \log ^2(2)-\frac{1}{120} 19 \log ^5(2)\\ &\qquad +\frac{1}{9} \pi ^2 \log ^3(2)-\frac{7}{180} \pi ^4 \log (2)
		\end{aligned}$$
		The case of weight $6$ and $7$ respectively give
		\small \begin{multline*}\zeta(\overline{5},1) = 8 \text{Li}_6\left(\frac{1}{2}\right)+3 \text{Li}_5\left(\frac{1}{2}\right) \log (2)+\frac{S_6}{2}+\frac{\zeta (3)^2}{2}-\frac{1}{6} \zeta (3) \log ^3(2)+\frac{19}{4} \zeta (5) \log (2)-\frac{\pi ^6}{112}\\ -\frac{19 \log ^6(2)}{1440}+\frac{1}{72} \pi ^2 \log ^4(2)-\frac{7}{720} \pi ^4 \log ^2(2) \end{multline*}
		\begin{multline*}\zeta(\overline{5},1,1) = \frac{1}{2} \log (2) \zeta(\overline{5},1)+\frac{11 \text{Li}_7\left(\frac{1}{2}\right)}{2}+\frac{3}{2} \text{Li}_6\left(\frac{1}{2}\right) \log (2)+\frac{S_7}{4}+\\ \frac{\pi ^4 \zeta (3)}{90}+\frac{\pi ^2 \zeta (5)}{6}-\frac{535 \zeta (7)}{64}+\frac{1}{48} \zeta (3) \log ^4(2)-\frac{19}{16} \zeta (5) \log ^2(2)-\frac{1}{4} \zeta (3)^2 \log (2)+\frac{19 \log ^7(2)}{20160}\\ -\frac{1}{720} \pi ^2 \log ^5(2)+\frac{7 \pi ^4 \log ^3(2)}{4320}+\frac{1}{224} \pi ^6 \log (2) \end{multline*}
	\end{example}
	
	\normalsize{where} $$S_n = \sum_{n=1}^\infty \frac{(-1)^n}{n^s 2^n} \binom{2n}{n}^{-1}$$ Note that $S_n$ converges quite fast, on geometric rate of $8^{-n}$, so the above two series are suitable for high precision ($> 10^4$ decimal digit) calculation of these two constants, this is better than conventional method on level $2$ CMZV that converges only at rate of $2^{-n}$ \cite[Sect.~7]{borwein2001special}. \\
	The above series for $\zeta(\bar{5},1)$ is originally due to Zhao \cite{ZhaoMingHao1}. We cannot resist to mention the beautiful closed-form of $\zeta(\bar{5},1)$ conjectured by Charlton \cite{charlton2019functional}, discovered via motivic techniques:
	\small \begin{multline*}\zeta(\bar{5},1) = -\frac{126 \text{Li}_6\left(\frac{1}{2}\right)}{13}-\frac{162 \text{Li}_6\left(-\frac{1}{2}\right)}{13}+\frac{\text{Li}_6\left(-\frac{1}{8}\right)}{39}+\frac{3 \zeta (3)^2}{8}+\frac{31}{16} \zeta (5) \log (2)-\frac{1787 \pi ^6}{589680}-\frac{1}{208} \log ^6(2)\\ +\frac{1}{208} \pi ^2 \log ^4(2)-\frac{1}{156} \pi ^4 \log ^2(2)\end{multline*}
	
	\begin{theorem}
		For positive integer $s$,
		$$\sum_{n=1}^\infty \frac{1}{n^s 2^n \binom{3n}{n}} \in \textsf{CMZV}^4_s$$
	\end{theorem}
	\begin{proof}
		The summation equals to $$I = 2\int_0^1 \frac{1}{x}\Li_{s-1}\left(\frac{1}{2}x^2(1-x)\right) dx$$
		Note that
		$$\Li_{s-1}\left(\frac{1}{2}x^2(1-x)\right) = \int_0^x (2\omega(0)+\omega(1))^{s-2} (-\omega(-1) - \omega(1-i) -\omega(1+i))$$
		so $I$ will be the iterated integral of a combination of $\omega(c_1)\cdots \omega(c_s)$, with $c_i \in \{0,1,-1,1-i,1+i\}$ with $-1, 1+i, 1-i$ \textit{never occur in the same word}. If a word contains only $\omega(0),\omega(1),\omega(-1)$ then it is a level 2 CMZV, if a word contains only $\omega(0),\omega(1),\omega(1\pm i)$, then above proposition transforms it to a level 4 CMZV, so $I \in \textsf{CMZV}^4_s$. 
	\end{proof}
	
	\begin{example}
		\small $$\begin{aligned}
			\sum_{n=1}^\infty \frac{1}{n^3 2^n \binom{3n}{n}} &= \pi  C-\frac{33 \zeta (3)}{16}+\frac{\log ^3(2)}{6}-\frac{1}{24} \pi ^2 \log (2) \\ 
			\sum_{n=1}^\infty \frac{1}{n^4 2^n \binom{3n}{n}} &= 2 \pi  \Im\left(\text{Li}_3\left(\frac{1}{2}+\frac{i}{2}\right)\right)-\frac{21 \text{Li}_4\left(\frac{1}{2}\right)}{2}-\frac{57}{8} \zeta (3) \log (2)+\frac{61 \pi ^4}{960} -\frac{23}{48} \log ^4(2)+\frac{19}{48} \pi ^2 \log ^2(2) \\ 
			\sum_{n=1}^\infty \frac{1}{n^5 2^n \binom{3n}{n}} &= 4 \pi  \Im\left(\text{Li}_4\left(\frac{1}{2}+\frac{i}{2}\right)\right)+3 \pi  \beta(4)-\frac{51 \text{Li}_5\left(\frac{1}{2}\right)}{2} -15 \text{Li}_4\left(\frac{1}{2}\right) \log (2)+\frac{\pi ^2 \zeta (3)}{4}+\frac{9 \zeta (5)}{2}\\ &\qquad -3 \zeta (3) \log ^2(2) -\frac{97}{240} \log ^5(2)+\frac{41}{144} \pi ^2 \log ^3(2)-\frac{61}{960} \pi ^4 \log (2)
		\end{aligned}$$
	\end{example}
	
	The last two series were conjectured by Borwein \cite[p.~27-28]{borwein2004experimentation}. The penultimate series is solved by Zhao \cite{ZhaoMingHao1}, who relies on \textit{ad hoc} integration by parts and certain level $4$ polylogarithmic integrals. 
	
	\newpage
	\section*{Appendix A: Mathematica package}
	The package can be download at https://www.researchgate.net/publication/357601353. A short documentation on installation and funtionalities can also be found there. All explicit one-dimensional integrals that appear before Section 4 can be calculated by the package. \par 
	Currently it can only handle ordinary polylogarithm $\Li_n$ as the integrand, the case for generalized polylogarithm might be added in a future version.
	
	Here are lists of “new constants” that appear at each weight for level $2$ and $4$. For level $1$, consult \cite{Petitot}.
	\begin{table}[h]
		\begin{tabular}{|l|c|}
			\hline
			Weight & A basis of $\widetilde{\textsf{CMZV}}^2_w$ \\ \hline
			1 & $\log 2$ \\ \hline
			2 & $\zeta(2)$ \\ \hline
			3 & $\zeta(3)$ \\ \hline
			4 & $\Li_4(1/2)$ \\ \hline
			5 & $\Li_5(1/2), \zeta(5)$ \\ \hline
			6 & $\Li_6(1/2), \zeta(\bar{5},1)$ \\ \hline
			7 & $\Li_7(1/2), \zeta(7), \zeta(\bar{5},1,1), \zeta(5,\bar{1},1)$ \\ \hline
			8 & $\Li_8(1/2), \zeta(6,2), \zeta(\bar{7},1), \zeta(\bar{5},1,\bar{1},1), \zeta(\bar{5},\bar{1},\bar{1},\bar{1})$ \\ \hline
		\end{tabular}
	\end{table}
	
	\begin{table}[h]
		\begin{tabular}{|l|c|}
			\hline
			Weight & A basis of $\widetilde{\textsf{CMZV}}^4_w$ \\ \hline
			1 & $\log 2, i\pi$ \\ \hline
			2 & $iC$ \\ \hline
			3 & $\zeta(3), i\Im\Li_3((1+i)/2)$ \\ \hline
			4 & $\beta(4), \Li_4(1/2), i\Im\Li_4((1+i)/2)$ \\ \hline
			5 & $\zeta(5), \Li_5(1/2), i\Im\Li_5((1+i)/2), L_{4,1}(i,1), L_{4,1}(i,-1), L_{3,1,1}(1,1,i)$ \\ \hline
		\end{tabular}
	\end{table}
	
	\newpage
	\section*{Appendix B: Admissible $4$-rational functions}
	If $R(x)$ is $N$-admissible, then so are 
	$$\{R, 1-R, \frac{R}{R-1}, \frac{1}{R}, \frac{R-1}{R}, \frac{1}{1-R}\}$$
	this amounts to an $S_3$-action. (with $S_3$ symmetric group on $3$ letters). \par 
	
	When $N=4$, the automorphism group of $\hat{\mathbb{C}}$ that permutes $\{0,\infty,\pm i, \pm 1\}$ is the (orientation preserving) octahedral group $S_4$. So if $R(x)$ is $4$-admissible, then for $(g,h)\in S_3\times S_4$, $gR(h^{-1}x)$ is also $4$-admissible. This defines an $S_3\times S_4$ action on the set of $4$-admissible functions. 
	
	\begin{table}[h]
		\begin{tabular}{|l|c|}
			\hline
			$R(x)$ & Size of orbit \\ \hline
			$x$ & 72 \\ \hline
			$x^2$ & 36 \\ \hline
			$(x^2+1)/(2x)$ & 36 \\ \hline
			$x^4$ & 18 \\ \hline
			$4x^2/(1+x^2)^2$ & 6 \\ \hline
		\end{tabular}
		\caption{\small Orbits of $4$-admissible functions known to the author, there might be more}
	\end{table}
	
	The following Mathematica code finds all distinct elements in an orbit:
	\begin{verbatim}Clear[S4, f, g, x, S3S4orbit]; S4 = {x, (I - x)/(I + x), (-I - x)/(-I + x), (
			I + x)/(-I + x), (-I + x)/(I + x), -((I (-1 + x))/(1 + x)), (
			I (-1 + x))/(1 + x), (I (1 + x))/(-1 + x), -((I (1 + x))/(-1 + x)), 
			1/x, -x, -(1/x), I x, (1 - x)/(1 + x), (-1 - x)/(-1 + x), (
			1 + x)/(-1 + x), (-1 + x)/(1 + x), -((I (I + x))/(-I + x)), (
			I (I + x))/(-I + x), (I (-I + x))/(
			I + x), -((I (-I + x))/(I + x)), -(I/x), -I x, I/x}; 
		f[x_] := x/(x - 1); g[x_] := 1 - x; 
		S3S4orbit[rat_] := 
		DeleteDuplicatesBy[
		Flatten[{f[#], g[#], f[g[#]], g[f[#]], f[g[f[#]]], #} &[
		rat /. x -> #] & /@ S4] // Simplify // 
		Sort, # /. x -> 1/11 &];
	\end{verbatim}
	
	Copy the above code into Mathematica, then execute \begin{verbatim} S3S4orbit[x]\end{verbatim} gives $72$ distinct $4$-admissible functions. 
	
	\newpage

	\bibliographystyle{plain} 
	\bibliography{ref.bib} 
	
\end{document}